\documentclass[12pt,twoside]{article}

\usepackage{amsmath,amssymb,amsthm} 

\newcommand{\Norm}{\mathrm{Norm}}
\newcommand{\li}{\mathrm{li}}

\newcommand{\as}{\quad\text{as}\quad}
\newcommand{\tinf}{\to\infty}

\newcommand{\nt}{\notag}

\newcommand{\bC}{\mathbb{C}}
\newcommand{\bR}{\mathbb{R}}

\newcommand{\noi}{\noindent}
\renewcommand{\Re}{\mathrm{Re}}
\renewcommand{\Im}{\mathrm{Im}}

\newcommand{\divset}{\hspace{3pt}|\hspace{3pt}}
\newcommand{\bigdivset}{\hspace{3pt}\big|\hspace{3pt}}

\newtheorem{thm}{Theorem}[section]

\newtheorem{lem}[thm]{Lemma}

\newtheorem{cor}[thm]{Corollary}

\numberwithin{equation}{section}

\title{Prime and zero distributions for meromorphic Euler products}
\author{Yasufumi Hashimoto}
\date{}

\begin{document}
\markboth
{Y. Hashimoto}
{Prime and zero distributions}
\pagestyle{myheadings}

\maketitle

\begin{abstract}
The aim of the present paper is to study the relations between the prime distribution 
and the zero distribution for generalized zeta functions  
which are expressed by an Euler products and are analytically continued as 
meromorphic functions of finite order. 
In this paper, we give an inequality between the order of the zeta function as a meromorphic function 
and the growth of the multiplicity in the prime distribution. 
\end{abstract}

\section{Introduction}

The aim of the present paper is to study the relation between the distributions of 
``primes" and ``zeros" for generalized meromorphic zeta functions expressed by the Euler products. 

For the Riemann zeta function, the zero distribution problems such as the Riemann hypothesis, 
the multiplicity-one problem and the GUE conjecture, and the prime distribution problems such as 
estimating the error terms of the prime number theorem, the twin prime problem,  
and the prime distribution in short intervals have been studied by many mathematicians in long time 
(see, e.g. \cite{Ed}, \cite{Iv} and \cite{Ti}). 
As well-known, there are deep relations between the zero distribution and the prime distribution; 
in fact, the best possible estimate of error term 
of the prime number theorem ($O(x^{1/2+\epsilon})$) would be obtained if the Riemann hypothesis 
could be proved positively. 

On the other hand, for the Selberg zeta function defined by the Euler product over the length 
of the primitive closed geodesics on a volume finite Riemann surface (or a higher dimensional 
negatively curved locally symmetric Riemannian manifold), such problems are described 
as geometric and spectral problems because the ``primes" for Selberg's zeta function 
are the primitive closed geodesics and the non-trivial zeros are written by the eigenvalues of the Laplacian 
on the corresponding Riemann surface. 
The relation between the ``primes" and ``zeros" in this case is described by the Selberg trace formula, 
and its relation gives the analytic continuation of the Selberg zeta function. 

There are some common properties between Riemann's and Selberg's zeta functions 
such like analytic continuations to the whole complex plane plane as meromorphic functions, 
functional equations and so on. 
However, the prime and zero distributions for Selberg's zeta function 
are different to those for Riemann's zeta function, 
and the difference sometimes causes technical difficulties when one analyzes these distributions. 

One of such differences is that the order of Riemann's zeta function is one 
but that of Selberg's zeta function is two (or the dimension of the corresponding manifold). 
While the Riemann hypothesis for the Selberg zeta function almost holds, 
the best possible estimate of the error terms of the prime geodesic theorem, 
the prime number theorem of Selberg's zeta version, have been never obtained 
because the order is two. 

Another difference is in the prime distributions. 
Since the Riemann zeta function is written by the Euler product over prime numbers with multiplicity one, 
it has the Dirichlet series expression over integers and also the integral expression. 
Many analytic number theorists have used these expressions 
to analyze the prime and zero distributions. 
However, the Selberg zeta function is written by the Euler product 
over non-integer (and non-rational) values with unbounded 
multiplicities (see \cite{Ra}).
Then any good Dirichlet series expressions and integral expressions of it have never been found.

\begin{center}
\begin{tabular}{|c||c|c|} 
\hline
 & Riemann's zeta  & Selberg's zeta \\
\hline
prime distributions & multiplicity one & unbounded multiplicity \\
\hline
order of zeta (zero distrib.) & one & two \\
\hline
\end{tabular}
\end{center}

In the present paper, we study the relations between the prime and zero distributions 
for general zeta functions which are defined by the Euler product and 
have meromorphic continuations to the whole complex planes, 
in the view of such differences between the Riemann zeta function and the Selberg zeta function. 
Actually, we obtain the inequality between the order of the zeta function and 
the multiplicity in the prime distributions (Theorem \ref{thm}), which shows that the order increases 
as (the average of) the growth of the multiplicity does. 
As corollaries of the theorem, we give some properties of the zeta functions in Corollary \ref{cor} 
associated with the critical strips.

\section{Preliminaries and the main results}

Let $P$ be an infinite countable set and $N:P\to\bR_{>1}$ a map satisfying 
$\sum_{p\in P}N(p)^{-a}<\infty$ for some $a>0$. 
Assume that $a_P:=\inf\{a>0\divset \sum_{p\in P}N(p)^{-a}<\infty\}$ is positive 
and normalize $N$ as $a_P=1$. 
Define the zeta function of $P$ by 
\begin{align}\label{zeta}
\zeta_P(s):=\prod_{p\in P}(1-N(p)^{-s})^{-1}\quad \Re{s}>1,
\end{align}
and assume that   
(i) $\zeta_P(s)$ is non-zero holomorphic in $\{\Re{s}\geq 1\}$ without a simple pole at $s=1$, and 
(ii) $\zeta_P(s)$ can be analytically continued to the whole complex plane $\bC$ 
as a meromorphic function of finite order $d\geq0$. 

The assumption (i) implies that 
\begin{align}
\#\{p\in P\divset N(p)<x\}\sim \li(x)\as x\tinf,
\end{align}
where $\li(x):=\int_2^{x}(\log{t})^{-1}dt$ (see, e.g., \cite{Ku}).  
This formula is interpreted as a generalization of the prime number theorem. 
When we denote by $\Norm(P):=\{N(p)\divset p\in P\}$ 
and $m(N)$ the number of $p\in P$ with the norm $N(p)=N$ for $N\in\Norm(P)$,
the zeta function and the prime number theorem are written by
\begin{align}
&\zeta_P(s)=\prod_{N\in\Norm(P)}(1-N^{-s})^{-m(N)}\quad \Re{s}>1,\\
&\sum_{\begin{subarray}{c}N\in \Norm(P)\\ N<x\end{subarray}}m(N)\sim\li(x)\as x\tinf.
\end{align} 

On the other hand, the assumption (ii) implies that 
\begin{align}\label{zero}
T^{d-\epsilon} \ll \#\{\sigma\in\Lambda_P\divset |\sigma|<T\}\ll T^{d+\epsilon},
\end{align}
for any $\epsilon>0$. 
Here $\Lambda_P$ is the set of singular points of $\zeta_P(s)$ 
and the number of singular points above is counted with multiplicities.  
Put the number $d_1\in[0,d]$ such that 
\begin{align}\label{zero1}
\#\{\sigma\in\Lambda_P\divset T-1<|\sigma|<T+1\}\ll T^{d_1+\epsilon}.
\end{align} 

When $P$ is the set of rational prime numbers and $N(p)=p$, 
the zeta function $\zeta_P(s)$ is the Riemann zeta function 
\begin{align*}
\zeta(s)=\prod_{p}(1-p^{-s})^{-1}\quad \Re{s}>1,
\end{align*}
which satisfies the assumptions (i) and (ii) with $d=1$ and $d_1=0$ (see, e.g. \cite{Ti}). 
It is easy to see that $\Norm(P)=P$ and $m(p)=1$ for any $p\in P$.

On the other hand, when $P$ is the set of primitive hyperbolic conjugacy classes 
of a discrete subgroup $\Gamma$
of $\mathrm{SL}_2(\bR)$ which is the fundamental group of a volume finite Riemann surface 
with the hyperbolic metric 
and $N(p)$ is the square of the larger eigenvalue of $p$, 
the zeta function $\zeta_P(s)$ is the Selberg (Ruelle) zeta function. 
By virtue of Selberg's trace formula, we see that the Selberg zeta function satisfies 
the assumptions (i) and (ii) with $d=2$ and $d_1=1$ (see, e.g. \cite{He}). 
Also, it is known that $m(N)$ is unbounded (\cite{Ra}) 
and, furthermore for arithmetic $\Gamma$, 
it has been considered that the asymptotic distribution of $m(N)$ 
is close to $\li(N^{1/2})$ (see \cite{Sc} and \cite{H}). 

\begin{center}
\begin{tabular}{|c||c|c|} 
\hline
 & Riemann's zeta  & Selberg's zeta \\
\hline
m(N) & 1 (bounded) & unbounded \\
\hline
order & $d=1$ ($d_1=0$) & $d=2$ ($d_1=1$) \\
\hline
\end{tabular}
\end{center}

The main result in this paper is as follows.
\begin{thm}\label{thm}
Let $\rho\leq1$ and $1\leq \alpha\leq2$ be numbers respectively satisfying 
that $\zeta_P(s)$ has at most finite number of singular points in $\{\Re{s}\geq\rho\}$ 
and that 
\begin{align*}
x^{\alpha-\epsilon}\ll \sum_{\begin{subarray}{c}N\in\Norm(P)\\N<x\end{subarray}}m(N)^2\ll x^{\alpha+\epsilon}.
\end{align*}
Then we have 
\begin{align}
d+d_1\geq \max\Big\{1,\frac{2-2\rho}{2-\alpha}\Big\}.
\end{align}
\end{thm}

For the Riemann zeta function, since $d+d_1=1$ and $\alpha=1$, we have $\rho\geq1/2$. 
This is a well-known fact. 
Remark that if $\rho=1/2$ then the Riemann hypothesis would be true and 
if $\rho>1/2$ then it would be false. 
Such a situation is same for the Dedekind zeta functions (see, e.g. \cite{Na}).
On the other hand, for the Selberg zeta function associated with Riemann surfaces, 
since $d+d_1=3$ and $\rho=1/2$, we have $\alpha\leq 5/3$. 
This was proven in our previous work \cite{H}. 
Note that it has been expected that $\alpha\leq3/2$ for any volume finite Riemann surfaces 
and $\alpha=3/2$ for arithmetic surfaces. 
Then the inequality above might be able to be improved unconditionally or with some natural conditions.

We also note that the theorem above gives the following corollary.
\begin{cor}\label{cor} 
The zeta function $\zeta_P(s)$ has the following properties.\\
(1) The set of imaginary parts of the singular points of $\zeta_P(s)$ is unbounded ($\rho>-\infty$). \\
(2) If $d+d_1=1$ then $\rho\geq\alpha/2(\geq1/2)$.\\
(3) If $d+d_1=1$ and $\rho=1/2$ then $\alpha=1$.\\
(4) If $\alpha=2$ then $\rho=1$.
\end{cor}

In other words, we can say that  
(1) every $\zeta_P(s)$ have ``non-trivial" zeros (or poles), 
(2) $\zeta_P(s)$ of order $1$  has infinitely many singular points 
in $\{1/2-\epsilon<\Re{s}<1\}$ for any $\epsilon>0$, 
(3) the Riemann hypothesis is not true for $\zeta_P(s)$ of order $1$ and of $\alpha>1$, 
(4) if the prime distribution in $\zeta_P(s)$ is of very high multiplicity 
then $\zeta_P(s)$ has infinitely many singular points near $\Re{s}=1$.  

\section{Proof of the theorem}

Let $v\in C^{\infty}(\bR_{>0})$ be a function satisfying $0\leq v(x)\leq 1$ for any $x>0$ 
and 
\begin{align*}
v(x)=\begin{cases} 1 & (0<x\leq 1), \\ 0 & (x\geq\min\{N(p)\divset p\in P\}=:\hat{N}). \end{cases}
\end{align*}
For $w(z):=\int_0^{\infty}-v'(x)x^zdx$, it is easy to see that   
\begin{align}\label{w}
w(z)=O\big(|z|^{-n}e^{\max(\Re{z},0)}\big)\as |z|\tinf
\end{align}
holds for any $n\geq1$, where the implied constant depends on $n$. 
We first state the following lemma.
\begin{lem}\label{lem1}
Let $X>0$ be a large number. Then we have 
\begin{align*}
\sum_{p\in P,k\geq1}\log{N(p)}N(p)^{-ks}v\Big(\frac{N(p)^{k}}{X}\Big)
=\sum_{\sigma\in\Lambda_P}\delta(\sigma)\frac{X^{\sigma-s}-1}{\sigma-s}w(\sigma-s),
\end{align*}
where $\delta(\sigma)=1$ when $\sigma$ is a zero and $\delta(\sigma)=-1$ when $\sigma$ is a pole.
\end{lem}

\begin{proof}
We calculate the following integral.
\begin{align*}
J(X,s):=\frac{1}{2\pi i}\int_{\Re{z}=2}-\frac{\zeta_P'(z+s)}{\zeta_P(z+s)}\frac{X^z-1}{z}w(z)dz.
\end{align*}
We first get 
\begin{align}
J(X,s)=&\sum_{p\in P,k\geq1}\log{N(p)}N(p)^{-ks}
\frac{1}{2\pi i}\int_{\Re{z}=2}\Big\{\Big(\frac{X}{N(p)^k}\Big)^z-N(p)^{-kz}\Big\}
\frac{w(z)}{z}dy\nt\\
=&\sum_{p\in P,k\geq1}\log{N(p)}N(p)^{-ks}\Big\{v\Big(\frac{N(p)^k}{X}\Big)-v(N(p)^k)\Big\}\nt\\
=&\sum_{p\in P,k\geq1}\log{N(p)}N(p)^{-ks}v\Big(\frac{N(p)^k}{X}\Big).\label{left1}
\end{align}

On the other hand, by using the residue theorem, we have
\begin{align}\label{right1}
J(X,s)=\sum_{\sigma\in\Lambda_P}\delta(\sigma)\frac{X^{\sigma-s}-1}{\sigma-s}w(\sigma-s).
\end{align}
Then the claim of the lemma follows immediately from \eqref{left1} and \eqref{right1}.
\end{proof}

For simplicity, we express the formula in Lemma \ref{lem1} as $G(s,X)=I(s,X)$. 
Let $u>1$ be a number satisfying that
there exists a constant $c>0$ such that $\zeta_P(s)$ has no singular points 
in $\{|\Re{s}+u|<c\}$.  
Taking the integrals $\int_{-u+iT}^{-u+2iT}|*|^2ds$ of the both hand sides of the formula,
we get the following lemmas.

\begin{lem}\label{lem2}
Let $T>0$ be a large number and 
\begin{align*}
\phi_u(X):=&\sum_{\begin{subarray}{c}p_1,p_2\in P\\k_1,k_2\geq1\\ N(p_1)^{k_1}=N(p_2)^{k_2}\end{subarray}}
\frac{k_1}{k_2}(\log{N(p_1)})^2N(p_1)^{2ku}v\Big(\frac{N(p_1)^{k_1}}{X}\Big)^2.
\end{align*}
Then we have 
\begin{align*}
\int_{T}^{2T}|G(-u+it,X)|^2dt=T\phi_u(X)+O(X^{2+2u}).
\end{align*}
\end{lem}

\begin{proof}
Directly calculating the integral, we have
\begin{align*}
\int_{T}^{2T}|G(-u+it,X)|^2dt
=&\sum_{\begin{subarray}{c}p_1,p_2\in P\\ k_1,k_2\geq1\end{subarray}}
\log{N(p_1)}\log{N(p_2)}N(p_1)^{k_1u}N(p_2)^{k_2u}\\
&\times v\Big(\frac{N(p_1)^{k_1}}{X}\Big)v\Big(\frac{N(p_2)^{k_2}}{X}\Big).
\int_{T}^{2T}\Big(\frac{N(p_2)^{k_2}}{N(p_1)^{k_1}}\Big)^{it}dt.
\end{align*}
Divide the sum above as follows.
\begin{align*}
\sum_{\begin{subarray}{c}p_1,p_2\in P\\ k_1,k_2\geq1\end{subarray}}
=\sum_{N(p_1)^{k_1}=N(p_2)^{k_2}}+\sum_{N(p_1)^{k_1}\neq N(p_2)^{k_2}}=:S_1+S_2.
\end{align*}
It is easy to see that the first sum $S_1$ is written by $S_1=T\phi_u(X)$.
Next we estimate $S_2$.
\begin{align*}
|S_2|\leq&\sum_{N(p_1)^{k_1}<N(p_2)^{k_2}<\hat{N}X}
\log{N(p_1)}\log{N(p_2)}N(p_1)^{k_1u}N(p_2)^{k_2u}\\
&\times \frac{\sin{\big(2T\log{\big(N(p_2)^{k_2}/N(p_1)^{k_1}\big)}\big)}
-\sin{\big(T\log{\big(N(p_2)^{k_2}/N(p_1)^{k_1}\big)}\big)}}
{\log{\big(N(p_2)^{k_2}/N(p_1)^{k_1}\big)}}\\
\leq&\sum_{N(p_1)^{k_1}<N(p_2)^{k_2}<\hat{N}X}\log{N(p_1)}\log{N(p_2)}
\frac{N(p_1)^{k_1(u+1)}N(p_2)^{k_2u}}{N(p_2)^{k_2}-N(p_1)^{k_1}}.
\end{align*}
We furthermore divide the sum above as follows.
\begin{align*}
\sum_{N(p_1)^{k_1}<N(p_2)^{k_2}<\hat{N}X}
=\sum_{N(p_1)^{k_1}<N(p_2)^{k_2}<\min{\big(2N(p_1)^{k_1},\hat{N}X\big)}}
+\sum_{2N(p_1)^{k_1}<N(p_2)^{k_2}<\hat{N}X}=:S_{21}+S_{22}.
\end{align*}
The later sum is estimated by 
\begin{align}
S_{22}=&\sum_{N(p_1)^{k_1}<\hat{N}X}\log{N(p_1)}\sum_{2N(p_1)^{k_1}<N(p_2)^{k_2}<\hat{N}X}
\log{N(p_2)}\frac{N(p_1)^{k_1(u+1)}N(p_2)^{k_2u}}{N(p_2)^{k_2}-N(p_1)^{k_1}}\nt\\
\leq &\sum_{N(p_1)^{k_1}<\hat{N}X}\log{N(p_1)}N(p_1)^{k_1u}
\sum_{N(p_2)^{k_2}<\hat{N}X}\log{N(p_2)}N(p_2)^{k_2u}=O(X^{2+2u}).
\end{align}
We estimate the former sum as follows.
\begin{align}
|S_{21}|\leq&\sum_{N(p_1)^{k_1}<\hat{N}X}\log{N(p_1)}
\sum_{N(p_1)^{k_1}<N(p_2)^{k_2}<2N(p_1)^{k_1}}
\log{N(p_2)}\frac{N(p_1)^{k_1(u+1)}N(p_2)^{k_2u}}{N(p_2)^{k_2}-N(p_1)^{k_1}}\nt\\
\leq &\sum_{N(p_1)^{k_1}<\hat{N}X}\log{N(p_1)}O\big(N(p_1)^{k_1(1+2u)}\big)
=O(X^{2+2u}).
\end{align}
This completes Lemma \ref{lem2}.
\end{proof}

\begin{lem}\label{lem3}
Let $U>0$ be a large number such that $U=o(T)$ as $T\tinf$ and $\rho$ a constant 
such that $0<\rho\leq 1$ and $\zeta_P(s)$ has at most finite number of singular points 
in $\{\Re{s}\geq\rho\}$.
Then, for any $n\geq1$, we have
\begin{align*}
\int_{T}^{2T}|I(-u+it,X)|^2dt
&=O(T^{-2n+1}X^{2+2u})+O(T^{-n+d+1+\epsilon}X^{1+\rho+2u})\\
&+O\big(T^{2d+\epsilon}U^{-n}X^{2(\rho+u)}\big)+O\big(T^{d+d_1+\epsilon}UX^{2(\rho+u)}\big),
\end{align*}
where the implied constants depend on $n$.
\end{lem}

\begin{proof}
Let $T':=\max\{|\Im{\sigma}| \bigdivset \sigma\in\Lambda_P,\Re{\sigma}\geq\rho\}$ 
and $T'':=10T'+1000$. 
We denote by $\Lambda'$  the set of $\sigma\in\Lambda_P$ satisfying $|\Im{\sigma}|>T''$ 
and $\Re{\sigma}>-T$. 
According to \eqref{w}, we have
\begin{align*}
|I(-u+it,X)|&\leq
\sum_{\begin{subarray}{c} \sigma\in\Lambda_P \\ \sigma=\alpha+i\beta \end{subarray}}
O\Big(X^{\max{(\alpha+u,0)}}|u+\alpha+i(\beta- t)^{-n}\Big)\\
&=\sum_{\begin{subarray}{c} \sigma\in\Lambda' \\ \sigma=\alpha+i\beta \end{subarray}}
O\Big(X^{\rho+u}|\delta+i(\beta-t)|^{-n}\Big)+O(X^{1+u}|t- T''|^{-n})\\
&=:I_1(t,X)+I_2(t,X).
\end{align*}
Take $T\gg T''$.
It is easy to see that 
\begin{align}\label{i22}
&\int_{T}^{2T}|I_2(t,X)|^2dt=O(T^{-2n+1}X^{2+2u}).
\end{align}

Next, we have 
\begin{align}\label{i12}
\int_T^{2T}\Re\big\{I_1(t,X)I_2(t,X)\big\}dt
=&\sum_{\sigma=\alpha+i\beta}
\int_T^{2T}O\big(X^{1+\rho+2u}|t|^{-n}|c+i(\beta-t)|^{-n}\big)dt\nt\\
=&\sum_{T/2<\beta<3T}O(T^{-n+1}X^{1+\rho+2u})
+\sum_{\text{other $\sigma$}}O(\beta^{-n}T^{-n+1}X^{1+\rho+2u})\nt\\
=&O(T^{-n+d+1+\epsilon}X^{1+\rho+2u}).
\end{align}

The remaining part of the proof is the estimation of the following integral.
\begin{align*}
&\int_T^{2T}|I_1(t,X)|^2dt=
\sum_{\begin{subarray}{c} \sigma_1=\alpha_1+i\beta_1 \\ \sigma_2=\alpha_2+i\beta_2 \end{subarray}}
\int_T^{2T}
O\Big(X^{2(\rho+u)}|c+i(\beta_1- t)|^{-n}|c-i(\beta_2-t)|^{-n}\Big)dt.
\end{align*}
Divide the sum above as follows.
\begin{align*}
\sum_{\beta_1,\beta_2}
=\sum_{\begin{subarray}{c}T/2<\beta_1,\beta_2<3T\\|\beta_1-\beta_2|<U\end{subarray}}
+\sum_{\begin{subarray}{c}T/2<\beta_1,\beta_2<3T\\|\beta_1-\beta_2|\geq U\end{subarray}}
+\sum_{\text{other $\sigma_1,\sigma_2$}}=:L_1+L_2+L_3.
\end{align*}
It is easy to see that $L_3=O(X^{2(\rho+u)})$. 
The second sum $L_2$ is estimated by 
\begin{align}
|L_2|=&\sum_{\begin{subarray}{c}T/2<\beta_1,\beta_2<3T\\|\beta_1-\beta_2|\geq U\end{subarray}}
\int_T^{2T}
O\Big(X^{2(\rho+u)}|2c+i(\beta_1- \beta_2)|^{-n}\times\nt\\ 
&\times\big(|c+i(\beta_1- t)|^{-1}+|c-i(\beta_2-t)|^{-1}\big)^{n}\Big)dt
=O(T^{2d+\epsilon}U^{-n}X^{2(\rho+u)}).
\end{align}
We can estimate the first sum $L_1$ by 
\begin{align}
|L_1|=&\sum_{\begin{subarray}{c}T/2<\beta_1,\beta_2<3T\\|\beta_1-\beta_2|<U\end{subarray}}
O(X^{2(\rho+u)})=O(T^{d+d_1+\epsilon}UX^{2(\rho+u)}).
\end{align}
Then we get the desired result.
\end{proof}

\noi \textit{Proof of Theorem \ref{thm}.} 
Due to Lemma \ref{lem2} and \ref{lem3}, we have
\begin{align*}
\phi_u(X)=O(T^{-1}X^{2+2u})+O(T^{2d-1+\epsilon}U^{-n}X^{2(\rho+u)})
+O(T^{d+d_1-1+\epsilon}UX^{2(\rho+u)}).
\end{align*}
Put $U=T^{\frac{d-d_1}{n+1}}$ into the above. 
Since $n$ can be taken arbitrary, we get
\begin{align*}
\phi_u(X)=O(T^{-1}X^{2+2u})+O(T^{d+d_1-1+\epsilon}X^{2(\rho+u)}).
\end{align*} 

If $d+d_1<1$ then we have $\phi_0(X)=o(1)$ as $X\tinf$ by taking $T$ sufficiently larger than $X$, 
for example $T=e^X$. However this contradicts to 
\begin{align}\label{lower}
\phi_u(X)\gg\sum_{\begin{subarray}{c}N\in\Norm(P)\\ N<X\end{subarray}}m(N)^2(\log{N})^2N^{2u}
\gg X^{\alpha+2u-\epsilon}.
\end{align} 
Thus $d+d_1\geq1$.

Taking $T=X^{\frac{1-\rho}{d+d_1}}$, we have 
\begin{align}\label{upper}
\phi_u(X)=O\Big(X^{2-\frac{2(1-\rho)}{d+d_1}+2u+\epsilon}\Big).
\end{align}
Combining \eqref{lower} and \eqref{upper}, we can easily obtain 
\begin{align}\label{result}
d+d_1\geq\frac{2-2\rho}{2-\alpha}.
\end{align}
This completes the proof of the theorem. 
The corollaries can be obtained easily from \eqref{result}.
\qed

\noindent 
\text{HASHIMOTO, Yasufumi}\\ 
Institute of Systems and Information Technologies/KYUSHU,\\
7F 2-1-22, Momochihama, Fukuoka 814-0001, JAPAN\\
e-mail:hasimoto@isit.or.jp

\end{document}